\documentclass[12pt,a4paper]{amsart}

\usepackage{lscape}
\usepackage{amscd}
\usepackage{amssymb}
\usepackage{amsmath}

\def\so{\mathfrak{so}}
\def\SO{\mathrm{SO}}
\def\sp{\mathfrak{sp}}
\def\Sp{\mathrm{Sp}}
\def\su{\mathfrak{su}}
\def\u{\mathfrak{u}}
\def\hh{\mathfrak{h}}
\def\tt{\mathfrak{t}}
\def\mm{\mathfrak{m}}
\def\gg{\mathfrak{g}}
\def\SU{\mathrm{SU}}
\def\Cl{\mathrm{Cl}}
\def\spin{\mathfrak{spin}}
\def\Spin{\mathrm{Spin}}

\def\G{\mathrm{G}}
\def\U{\mathrm{U}}
\def\T{\mathrm{T}}

\def\f{\varphi}

\def \RM{\mathbb{R}}

\def \CM{\mathbb{C}}
\def \TM{\mathbb{T}}
\def \HM{\mathbb{H}}
\def \SM{\mathbb{S}}
\def\HP{{\mathbb{H}\mathrm{P}}}
\def\CP{{\mathbb{C}\mathrm{P}}}
\def\E{\mathrm{E}}

\def\A{\mathcal{A}}
\def\G{\mathrm{G}}

\def\1{\mathbf{1}}
\def\#{\sharp}
\def\ad{\mathrm{ad}}

\def\C{\mathbb{C}}

\def\End{\mathrm{End}\,}
\def\g{\mathfrak{g}}
\def\h{\mathfrak{h}}
\def\H{\mathrm{H}}

\def\id{\mathrm{id}}

\def\l{\lambda}
\def\m{\mathfrak{m}}

\def\R{\mathcal{R}}
\def\rk{\mathrm{rk}}

\def\S{\mathrm{\Sigma}}

\def\W{\mathcal{W}}

\def\<#1,#2>{\langle\,#1,\,#2\,\rangle}

\def\beq{\begin{equation}}
\def\eeq{\end{equation}}

\def\norm(#1){\|#1\|^2}
\def\rectangle(#1,#2)[#3,#4]#5{
 \multiput(#1,#2)(#3,0)2{\line(0,1){#4}}\multiput(#1,#2)(0,#4)2{\line(1,0){#3}}
 \put(#1,#2){\vbox to #4pt{\hbox to #3pt{\hfill}\vfill}}}
\def\recttext(#1,#2)[#3,#4]#5{\put(#1,#2)
 {\vbox to #4pt{\vfill\hbox to #3pt{\hss#5\hss}\vfill}}}
\def\tfrac#1#2{{\textstyle\frac{#1}{#2}}}
\newtheorem{Lemma}{Lemma}[section]
\newtheorem{Proposition}[Lemma]{Proposition}

\newtheorem{Theorem}[Lemma]{Theorem}

\theoremstyle{definition}

\newtheorem{Definition}[Lemma]{Definition}
\newtheorem{Remark}[Lemma]{Remark}

\title{Homogeneous almost quaternion-Hermitian manifolds}

\author{Andrei Moroianu, Mihaela Pilca, Uwe Semmelmann}
\thanks{This work was supported 
by the contract ANR-10-BLAN 0105 ``Aspects Conformes de la
G{\'e}om{\'e}trie''. The second-named author thanks
the Centre de Math{\'e}matiques de 
l'{\'E}cole Polytechnique for hospitality 
during the preparation of this work.}

\address{Andrei Moroianu \\ Universit\'e de Versailles-St Quentin \\
Laboratoire de Math\'ematiques \\ UMR 8100 du CNRS\\
45 avenue des \'Etats-Unis\\
78035 Versailles}
\email{andrei.moroianu@math.cnrs.fr}

\address{Mihaela Pilca\\Fakult\"at f\"ur Mathematik\\
Universit\"at Regensburg\\Universit\"atsstr. 31 
D-93040 Regensburg, Germany
\emph{and} 
Institute of Mathematics ``Simion Stoilow" of the Romanian Academy, 
21, Calea Grivitei Str.
010702-Bucharest, Romania}
\email{Mihaela.Pilca@mathematik.uni-regensburg.de}

\address{Uwe Semmelmann\\
Institut f\"ur Geometrie und Topologie \\
Fachbereich Mathematik\\
Universit{\"a}t Stuttgart\\
Pfaffenwaldring 57 \\
70569 Stuttgart, Germany
}
\email{uwe.semmelmann@mathematik.uni-stuttgart.de}

\begin{document}

\begin{abstract} An almost quaternion-Hermitian structure on a Riemannian manifold $(M^{4n},g)$ is
a reduction of the structure group of $M$ to $\Sp(n)\Sp(1)\subset \SO(4n)$. In this paper we show
that a compact simply connected homogeneous almost quaternion-Hermitian manifold of non-vanishing
Euler characteristic is either a Wolf space, or $\SM^2\times\SM^2$, or the complex quadric 
$\SO(7)/\U(3)$.

\medskip

\noindent 2010 {\it Mathematics Subject Classification}: Primary: 53C30, 53C35, 53C15. Secondary: 17B22

\smallskip

\noindent {\it Keywords}: Quaternion-Hermitian structures, homogeneous spaces, root systems, Clifford structures.

\end{abstract}
\maketitle

\section{Introduction}

The notion of (even) Clifford structures on Riemannian manifolds was introduced in \cite{ms}.
Roughly speaking, a rank $r$ (even) Clifford structure on $M$ is a rank $r$ Euclidean bundle whose
(even) Clifford algebra bundle acts on the tangent bundle of $M$. For $r=3$, an even Clifford
structure on $M$ is just an almost quaternionic structure, i.e. a rank 3 sub-bundle $Q$ of the
endomorphism bundle $\End(\T M)$ locally spanned by three endomorphisms $I,J,K$ satisfying the
quaternionic relations
$$I^2=J^2=K^2=-\id,\qquad IJ=K.$$
If moreover $Q\subset \End^-(\T M)$ (or, equivalently, if $I,J,K$ are $g$-orthogonal), the
structure $(M,g,Q)$ is called almost quaternion-Hermitian \cite{ca04,cs04,cs08,sw97}.

Homogeneous even Clifford structures on homogeneous compact manifolds of non-vanishing Euler
characteristic were studied in \cite{mp}, where it is established an upper bound for their rank, as
well as a description of the limiting cases. In this paper we consider the other extremal case,
namely even Clifford structures with the lowest possible (non-trivial) rank, which is $3$ and give
the complete classification of compact homogeneous almost quaternion-Hermitian manifolds $G/H$ with
non-vanishing Euler characteristic. This last assumption turns out to be crucial at several places
throughout the proof (see below). Without it, the classification is completely out of reach, but
there are lots of homogeneous examples constructed for instance by D. Joyce \cite{j1,j2} and
O. Maci\'a \cite{ma10}. 

Our classification result is the following:
\begin{Theorem}\label{expl}
A compact simply connected homogeneous manifold $M=G/H$ of non-vanishing Euler characteristic
carries a homogeneous almost quaternion-Hermitian structure if and only if it belongs to the following list:
\begin{itemize}
\item Wolf spaces $G/N$ where $G$ is any compact simple Lie group and $N$ is the normalizer of some
subgroup $\Sp(1)\subset G$ determined by a highest root of $G$, cf. \cite{w}.
\item $\SM^2\times\SM^2$.
\item $\SO(7)/\U(3)$.
\end{itemize}
\end{Theorem}

Let us first give some comments on the above list. The Wolf spaces are quaternion-K\"ahler
manifolds \cite{w}, so they admit not only a topological but even a {\em holonomy} reduction to
$\Sp(n)\Sp(1)$. In dimension 4, every orientable manifold is almost quaternion-Hermitian since
$\Sp(1)\Sp(1)=\SO(4)$. In this dimension there exist (up to homothety) only two compact simply
connected homogeneous manifolds with non-vanishing Euler characteristic: $\SM^2\times\SM^2$ and
$\SM^4$. The latter is already a Wolf space since $\SM^4=\HP^1$, this is why in dimension 4, the
only extra space in the list is $\SM^2\times\SM^2$. Finally, the complex quadric
$\SO(7)/\U(3)\subset \CP^7$, which incidentally is also the twistor space of $\SM^6$, carries a
1-parameter family of $\Sp(3)\U(1)$ structures with fixed volume. Motivated by our present
classification, F. Mart\'\i n Cabrera and A. Swann \cite{cs12} are currently investigating the
quaternion Hermitian type of this family.

The outline of the proof of Theorem \ref{expl} is as follows: The first step is to show (in
Proposition \ref{simple}) that $G$ has to be a simple Lie group, unless $M=\SM^2\times\SM^2$. The
condition $\chi(M)\ne 0$ (which is equivalent to $\rk(H)=\rk(G)$) is used here in order to ensure
that every subgroup of maximal rank of a product $G_1\times G_2$ is itself a product. 
The next step is to rule out the case $G=\G_2$ which is the only simple group for which the ratio
between the length of the long and short roots is $\sqrt3$. 
Once this is done, we can thus assume that either all roots of $G$ have the same length, or the
ratio between the length of the long and short roots is $\sqrt2$. We further show that if $G/H$ is
symmetric, then $H$ has an $\Sp(1)$-factor, so $M$ is a Wolf space.

Now, since $\rk(H)=\rk(G)$, the weights of the (complexified) isotropy representation $\m^\CM$ can
be identified with a subset of the root system of $G$. We show that the existence of a homogeneous
almost quaternion-Hermitian structure on $G/H$ implies that the set of weights $\W(\m^\CM)$ can be
split into two distinct subsets, one of which is obtained from the other by a translation
(Proposition \ref{wisotrop} below).
Moreover, if $G/H$ is not symmetric, then $[\m,\m]\cap\m\ne0$, so
$(\W(\m^\CM)+\W(\m^\CM))\cap\W(\m^\CM)\ne\emptyset$. Putting all this information together we are
then able to show, using the properties of root systems, that there is one single isotropy weight
system satisfying these conditions, namely the isotropy representation of $\SO(7)/\U(3)$, whose
restriction to $\SU(3)$ is isomorphic to $\CM^3\oplus(\CM^3)^*$ and is therefore quaternionic.

\section{Preliminaries}

Let $M=G/H$ be a homogeneous space. Throughout this paper we make the following assumptions:
\begin{itemize}
\item $M$ is compact (and thus $G$ and $H$ are compact, too).
\item The infinitesimal isotropy representation is faithful (this is always the case after taking an
appropriate quotient of $G$)
\item $M$ has non-vanishing Euler characteristic: $\chi(M)\ne 0$, or, equivalently,
$\rk(H)=\rk(G)$. 
\item $M$ is simply connected. An easy argument using the exact homotopy sequence shows that by
changing the representation of $M$ as homogeneous space if necessary, one can assume that $G$ is
simply connected and $H$ is connected (see \cite{au} for example).
\end{itemize}

Denote by $\hh$ and $\gg$ the Lie 
algebras of $H$ and $G$ and by $\m$ the orthogonal complement of $\hh$ in $\gg$ 
with respect to some $\ad_\gg$-invariant scalar product on $\gg$. The restriction 
to $\mm$ of this scalar product defines a homogeneous Riemannian metric $g$ on $M$.

An almost quaternion-Hermitian structure on a Riemannian manifold $(M,g)$ is a three-dimensional
sub-bundle of the 
bundle of skew-symmetric endomorphisms $\mathrm{End}^{-}(\T M)$, which is locally spanned by three
endomorphisms satisfying the quaternion relations \cite{ca04, sw97}. In the case where $M=G/H$ is
homogeneous, such a structure is called homogeneous if this three-dimensional sub-bundle is defined
by a three-dimensional $H$-invariant summand of  the second exterior power of the isotropy
representation $\Lambda^2\m=\mathrm{End}^{-}(\mathfrak{m})$. For our purposes, we give the following
equivalent definition, which corresponds to the fact that an almost quaternion-Hermitian structure
is just a rank 3 even Clifford structure (cf. \cite{mp, ms}):

\begin{Definition}\label{defi}
A {\em homogeneous almost quaternion-Hermitian}  structure on the Riemannian homogeneous
space $(G/H,g)$ is an orthogonal representation $\rho:H\to\SO(3)$ and an
$H$-equivariant Lie algebra morphism $\varphi: \so(3)\to
\mathrm{End}^{-}(\mathfrak{m})$ extending to an algebra representation
of the even real Clifford algebra $\Cl^0_3$ on $\mathfrak{m}$.
\end{Definition}

The $H$-equivariance of the morphism $\varphi: \so(3)\to
\mathrm{End}^{-}(\mathfrak{m})$ is with respect to the following actions of  $H$: the action on $\so(3)$ 
is given by the composition of  the adjoint representation of $\SO(3)$
with $\rho$, and the action on $\mathrm{End}^{-}(\mathfrak{m})$ is the one induced 
by the isotropy representation $\iota$ of $H$. Since $\f$ extends to a representation of $\Cl^0_3\simeq \HM$ on $\m$,
the above definition readily implies the following result (see also \cite[Lemma 3.2]{mp} or \cite{sa82}):

\begin{Lemma}\label{iso}
The complexified isotropy representation $\iota_*$ on $\m^\CM$ is isomorphic to
the tensor product $\m^\CM=\H\otimes_\CM\E$, where $\H$ is defined by the composition $\mu:=\xi\circ\rho_*$ of $\rho_*$ with the spin representation $\xi$ of $\so(3)=\spin(3)=\sp(1)$ on $\HM$, and $\E$ is defined by the composition $\lambda:=\pi\circ\iota_*$ of the isotropy representation with the projection of $\h$ to the kernel of $\rho_*$.
\end{Lemma}

\section{The classification}

In this section we classify all compact simply connected homogeneous almost quaternion-Hermitian
manifolds $M=G/H$ with non-vanishing Euler characteristic.

We choose a common maximal torus of $H$ and $G$ and denote by $\tt\subset \h$ its Lie algebra. Then
the root system $\mathcal{R}(\mathfrak{g})\subset \tt^*$
is the disjoint union of the root system $\mathcal{R}(\mathfrak{h})$ and the set $\mathcal{W}$ 
of weights of the complexified isotropy representation of the homogeneous space $G/H$. This follows
from the fact that the isotropy
representation is given by the restriction to $H$ of the adjoint
representation of $\mathfrak{g}$. 

The weights of the complex spin representation of $\so(3)$ on $\S_3^\C\simeq\HM$
are $\W(\S_3^\C)=\left\{\pm \tfrac12 e_1\right\}$, where $e_1$ is some element of norm $1$ of the dual of some Cartan sub-algebra of
$\so(3)$. We denote by $\beta\in\tt^*$ the pull-back through $\mu$ of the vector $\frac{1}{2}e_1$ and by $\A:=\{\pm\alpha_1,\ldots,\pm\alpha_n\}\subset\tt^*$ the weights of the self-dual representation $\l$. By Lemma \ref{iso}, 
we obtain the following description of the weights of the isotropy representation of any homogeneous almost quaternion-Hermitian manifold $M=G/H$, which is a particular case of \cite[Proposition 3.3]{mp}:

\begin{Proposition}\label{wisotrop}
The set $\mathcal{W}:=\mathcal{W}(\mathfrak{m})$ of weights of the isotropy representation
is given by:
\begin{equation}\label{weightsiso}
\mathcal{W}=\{\varepsilon_i\alpha_i+\varepsilon\beta\}_{1\leq i\leq n; \varepsilon_i, \varepsilon\in\{\pm1\}}.
\end{equation}
\end{Proposition}

As an immediate consequence we have:

\begin{Lemma}\label{rho}
Let $(G/H,g,\rho,\f)$ be a homogeneous almost quaternion-Hermitian structure as in Definition \ref{defi}. Then the infinitesimal representation $\rho_*:\h\to\so(3)$ does not vanish. 
\end{Lemma}
\begin{proof}
Suppose for a contradiction that $\rho_*=0$. Then the $\h$-representation $\H$ defined in Lemma \ref{iso} is trivial, so $\beta=0$ and $\m^\CM=\E\oplus\E$. Every weight of the (complexified) isotropy representation appears then twice in the root system of $G$, which is impossible ({\em cf.} \cite[p. 38]{s}).
\end{proof}

Our next goal is to show that the automorphism group of a homogeneous almost quaternion-Hermitian manifold
is in general a simple Lie group:

\begin{Proposition}\label{simple}
If $G/H$ is a simply connected compact homogeneous almost quaternion-Hermitian manifold with non-vanishing Euler characteristic, then either $G$ is simple or $G=\SU(2)\times\SU(2)$ and $M=\SM^2\times\SM^2$.
\end{Proposition}

\begin{proof}
 We already know that $G$ is compact and simply connected. If $G$ is not simple, then $G=G_1\times G_2$ with $\dim(G_i)\ge 3$. Let $\g_i$ denote the Lie algebra of $G_i$, so that $\mathfrak{g}=\mathfrak{g}_1\oplus\mathfrak{g}_2$. By a classical result of Borel and Siebenthal (\cite[p. 210]{bs49}),  the Lie algebra of the subgroup $H$ splits as $\mathfrak{h}=\mathfrak{h}_1\oplus\mathfrak{h}_2$, where $\mathfrak{h}_i=\mathfrak{h}\cap\mathfrak{g}_i$. Correspondingly, the isotropy representation splits as $\mathfrak{m}=\mathfrak{m}_1\oplus\mathfrak{m}_2$, where $\mathfrak{m}_i$ is the isotropy representation of $\mathfrak{h}_i$ in $\mathfrak{g}_i$. 

Let $\{e_1,e_2,e_3\}$ be an orthonormal basis of $\so(3)$ and let us denote by $J_i:=\varphi(e_i)$, for $1\leq i\leq 3$. The $H$-equivariance of $\varphi$ implies that 
\beq\label{eq}\varphi(\rho_*(X)e_i)=[\mathrm{ad}_X, J_i], \qquad\forall \ X\in\mathfrak{h},\ 1\leq i\leq 3.\eeq

We claim that the representation $\rho_*$ does not vanish on $\mathfrak{h}_1$ or on $\mathfrak{h}_2$. Assume for instance that 
$\rho_*(\h_1)=0$. We express each endomorphism $J_i$ of $\mathfrak{m}=\mathfrak{m}_1\oplus\mathfrak{m}_2$  as $$J_i=\begin{pmatrix} A_i & B_i \\ C_i & D_i\end{pmatrix}.$$ For  every $X\in\h_1$,  \eqref{eq} shows that $\ad_X$ commutes with $J_i$. Expressing  $$\ad_X=\begin{pmatrix} \ad^{\g_1}_X & 0 \\ 0 & 0\end{pmatrix}$$ we get in particular that $\ad^{\g_1}_X\circ B_i=0$ for all $X\in\h_1$. On the other hand, since $\rk(\h_1)=\rk(\g_1)$, there exists no vector in $\m_1$ commuting with all  $X\in\h_1$, so $B_i=0$ and thus $C_i=-B_i^*=0$ for $1\le i\le 3$. However, this would imply that the map $\f_1:\so(3)\to\End^-(\m_1)$ given by $\f_1(e_i)=A_i$ for $1\le i\le 3$
is a homogeneous almost quaternion-Hermitian structure on $G_1/H_1$ with vanishing $\rho_*$, which contradicts Lemma \ref{rho}.
This proves our claim.

Now, since $\rho_*:\h\to\so(3)$ is a Lie algebra morphism, we must have in particular $$[\rho_*(\h_1),\rho_*(\h_2)]=0.$$ By changing the orthonormal basis $\{e_1,e_2,e_3\}$ if necessary, we thus may assume that $\rho_*(\mathfrak{h}_1)=\rho_*(\mathfrak{h}_2)=\langle e_1\rangle$.
The Lie algebras $\h_1$ and $\h_2$ decompose as $\h_i=\h'_i\oplus \langle X_i\rangle$ where $\h'_i:=\ker(\rho_*)\cap\h_i$ and
$\rho_*(X_i)=e_1$ for $1\le i\le 2$. 

From \eqref{eq}, the following relations hold:
\begin{equation}\label{equivrel}
[\mathrm{ad}_{X_i}, J_2]=J_3, \quad [\mathrm{ad}_{X_i}, J_3]=-J_2,\qquad 1\leq i\leq 2.
\end{equation}
Like before we can write 
$$\ad_{X_1}=\begin{pmatrix} \ad^{\g_1}_{X_1} & 0 \\ 0 & 0\end{pmatrix},\qquad \ad_{X_2}=\begin{pmatrix}0 & 0 \\ 0 & \ad^{\g_2}_{X_2}\end{pmatrix},$$
so \eqref{equivrel} implies that $A_2=A_3=0$ and $D_2=D_3=0$. In particular $$-1=J_2^2=\begin{pmatrix} 
0 & B_2 \\
C_2 & 0                                                                                                                                                                                                                              \end{pmatrix}^2=\begin{pmatrix}
B_2C_2 & 0\\
0 & C_2B_2 \end{pmatrix},$$ thus showing that $B_2$ defines an isomorphism between $\m_2$ and $\m_1$ (whose inverse is $-C_2$).

On the other hand, since by \eqref{eq} $\ad_{X}$ commutes with $J_2$ for all $X\in \h'_1$, we obtain as before that $\ad^{\g_1}_X\circ B_2=0$ for all $X\in\h'_1$. Since $B_2$ is onto, this shows that the isotropy representation of $G_1/H_1$ restricted to $\h'_1$ vanishes, so $\h'_1=0$ and similarly $\h'_2=0$. We therefore have $\h_1=\h_2=\RM$, and since $\rk(G_i)=\rk(H_i)=1$, we get $\g_1=\g_2=\su(2)$.
We thus have $G=\SU(2)\times\SU(2)$, and $H=\TM^2$ is a maximal torus, so $M=\SM^2\times\SM^2$.
\end{proof}

We are in position to complete the proof of our main result:

\begin{proof}[Proof of Theorem \ref{expl}]
By Proposition~\ref{simple} we may assume that $G$ is simple. We first study the case $G=\G_2$ (this is the only simple group for which the ratio between the length of long and short roots is neither 1, nor $\sqrt 2$). 
The only connected subgroups of rank $2$ of $\G_2$ are $\U(2), \SU(3), \SO(4)$ and $\TM^2$.
The spaces $\G_2/\U(2)$ and $\G_2/\SU(3)$ have dimension 10 and 6 respectively, therefore they can not carry almost quaternion-Hermitian structures.

The quotient $\G_2/\SO(4)$ is a Wolf space, so it remains to study the generalized flag manifold $\G_2/\TM^2$. We claim that this space has no homogeneous almost quaternion-Hermitian structure. Indeed, if this were the case, using Proposition \ref{wisotrop} one could express the root system of $\G_2$ as the disjoint union of two subsets $$\W^+:=\{\varepsilon_i\alpha_i+\beta\}_{1\leq i\leq 3; \varepsilon_i\in\{\pm1\}},\qquad \W^-:=\{\varepsilon_i\alpha_i-\beta\}_{1\leq i\leq 3; \varepsilon_i\in\{\pm1\}}$$ such that there exists some vector $v\ (:=2\beta)$ with $\W^+=v+\W^-$. On the other hand, it is easy to check that there exist no such partition of $\R(\G_2)$.

Consider now the case where $M=G/H$ is a symmetric space. If $M$ is a Wolf space there is nothing to prove, so assume from now on that this is not the case. The Lie algebra of $H$ can be split as $\h=\ker(\rho_*)\oplus \h_0$, where $\h_0$ denotes the orthogonal complement of $\ker(\rho_*)$. Clearly $\h_0$ is isomorphic to $\rho_*(\h)\subset\so(3)$ so by Lemma \ref{rho},
$\h_0=\u(1)$ or $\h_0=\sp(1)$. The latter case can not occur since our assumption that $M$ is not a Wolf space implies that $\h$ has no $\sp(1)$-summand. We are left with the case when $\h=\ker(\rho_*)\oplus\u(1)$. We claim that this case can not occur either. Indeed, if such a space would carry a homogeneous almost quaternion-Hermitian structure, then the representation of $\ker(\rho_*)$ on $\m$ would be quaternionic. Two anti-commuting complex structures $I,J$ of $\m$ induce non-vanishing elements $a_I,a_J$ in the center of $\ker(\rho_*)$ (see the proof of \cite[Lemma 2.4]{ff}). On the other hand, the adjoint actions of $a_I$ and $a_J$ on $\m$ are proportional to $I$ and $J$ respectively (\cite[Eq. (4)]{ff}) and thus anti-commute, contradicting the fact that $a_I$ and $a_J$ commute (being central elements).

We can assume from now on, that $M=G/H$ is non-symmetric, $G$ is simple and $G\neq \G_2$. Up to a rescaling of the $\ad_G$-invariant metric on $\g$, we may thus assume that all roots of $\mathfrak{g}$ have square length equal to $1$ or $2$.

From \eqref{weightsiso}, it follows that 
$$\mathcal{R}(\mathfrak{g})=\mathcal{W}(\m)\cup \mathcal{R}(\mathfrak{h})=\{\varepsilon_i\alpha_i+\varepsilon\beta\}_{1\leq i\leq n; \varepsilon_i, \varepsilon\in\{\pm1\}}\cup \mathcal{R}(\mathfrak{h}).$$

Up to a change of signs of the $\alpha_i$'s, we may assume: 
\begin{equation}\label{signconv}
\<\beta,\alpha_i>\geq 0, \quad  \text{ for all } 1\leq i\leq n.
\end{equation}
Then either the roots $\beta+\alpha_i$ and $\beta-\alpha_i$ of $G$ have the same length, or
$|\beta+\alpha_i|^2=2$ and $|\beta-\alpha_i|^2=1$.
This shows that for each $1\leq i\leq n$,
\begin{equation}\label{pscalval}
\<\beta,\alpha_i>\in\left\{0,\frac{1}{4}\right\}.
\end{equation}
From the general property of root systems \eqref{normscal} below, it follows that:
\begin{equation}\label{betaval}
|\beta|^2\in\left\{\frac{1}{4},\frac{3}{4}, \frac{5}{4}\right\}.
\end{equation}

Since the homogeneous space $G/H$ is not symmetric, we have $[\mathfrak{m}, \mathfrak{m}]\not\subseteq \mathfrak{h}$, so there exist subscripts $i,j,k\in\{1, \dots, n\}$ such that $(\pm\beta\pm \alpha_i)+(\pm\beta\pm \alpha_j)=\pm\beta\pm \alpha_k$. Taking \eqref{pscalval} into account, we need to check the following possible cases (up to a permutation of the subscripts):
\begin{enumerate}
\item[a)] $\beta=\pm 2\alpha_1\pm\alpha_2$.
\item[b)] $\beta=\frac{\alpha_1}{3}$.
\item[c)] $\beta=\frac{\alpha_1\pm \alpha_2 \pm \alpha_3}{3}$.
\item[d)] $\beta= \alpha_1\pm \alpha_2 \pm \alpha_3$.
\end{enumerate}

We will show that cases a), b) and c) can not occur and that in case d) there is only one solution.

\begin{enumerate}
\item[a)] If $\beta=2\alpha_1+\alpha_2$, then $\beta+\alpha_2=2(\beta-\alpha_1)$ and this would imply the existence of two proportional roots, $\beta+\alpha_2$ and $\beta-\alpha_1$, in $\mathcal{W}\subseteq\mathcal{R}(\mathfrak{g})$, contradicting the property R2 of root systems (\emph{cf.} Definition~\ref{sysroot}). 
For all the other possible choices of signs in a) we obtain a similar contradiction.

\item[b)] If $\beta=\frac{\alpha_1}{3}$, then there exist two proportional roots: $\beta+\alpha_1=-2(\beta-\alpha_1)$ in $\mathcal{R}(\mathfrak{g})$, which again contradicts R2.

\item[c)] If $\beta=\frac{\alpha_1\pm \alpha_2 \pm \alpha_3}{3}$, then $|\beta|^2=\frac{1}{3}(\<\beta,\alpha_1>\pm \<\beta,\alpha_2>\pm \<\beta,\alpha_3>)$. From \eqref{pscalval} and \eqref{betaval}, it follows that the only possibility is:
\[\beta=\frac{\alpha_1+\alpha_2+\alpha_3}{3}, |\beta|^2=\frac{1}{4} \text{ and } \<\beta,\alpha_i>=\frac{1}{4}, \text{ for } 1\leq i\leq 3.\]
Together with \eqref{normscal}, this implies that for each $1\leq i\leq 3$ we have: $|\beta+\alpha_i|^2=2$, $|\beta-\alpha_i|^2=1$ and $|\alpha_i|^2=\frac{5}{4}$. Thus, for all $1\leq i,j\leq 3$, $i\neq j$, we have: $$\<\beta+\alpha_i,\beta-\alpha_j>=\frac{1}{4}-\<\alpha_i,\alpha_j>,$$ which by \eqref{normscal} must be equal to $0$ or $\pm1$, showing that $\<\alpha_i,\alpha_j>\in \{-\frac{3}{4}, \frac{1}{4}, \frac{5}{4}\}$. 
On the other hand, a direct computation shows that
\[\<\alpha_1,\alpha_2>+\<\alpha_1,\alpha_3>+\<\alpha_2,\alpha_3>=\frac{1}{2}\left(9|\beta|^2-\frac{15}{4}\right)=-\frac{3}{4},\]
which is not possible for any of the above values of the scalar products, yielding a contradiction.
\end{enumerate}

d) From \eqref{pscalval} and \eqref{betaval}, it follows that there are three possible sub-cases:
\begin{itemize}
\item[Case 1.] $\beta=\alpha_1\pm\alpha_2\pm\alpha_3, |\beta|^2=\frac{1}{4}, \<\beta,\alpha_1>=\frac{1}{4}, \<\beta,\alpha_2>=\<\beta,\alpha_3>=0$.
\item[Case 2.] $\beta=\alpha_1+\alpha_2-\alpha_3, |\beta|^2=\frac{1}{4}, \<\beta,\alpha_i>=\frac{1}{4}, 1\leq i\leq 3$.
\item[Case 3.] $\beta=\alpha_1+\alpha_2+\alpha_3, |\beta|^2=\frac{3}{4}, \<\beta,\alpha_i>=\frac{1}{4}, 1\leq i\leq 3$.
\end{itemize}

{\bf Case 1.} From \eqref{normscal} it follows $|\alpha_1|^2=\frac{5}{4}$ and $|\alpha_2|^2=|\alpha_3|^2=\frac{3}{4}$.
Since $\<\beta+\alpha_1,\beta-\alpha_1>=-1$ and $|\beta+\alpha_1|^2=2|\beta-\alpha_1|^2$, the reflexion property \eqref{newroots} shows that $2\beta=(\beta+\alpha_1)+(\beta-\alpha_1)$ and $3\beta-\alpha_1=(\beta+\alpha_1)+2(\beta-\alpha_1)$ belong to  $\mathcal{R}(\mathfrak{g})$. We show that these roots actually belong to $\mathcal{R}(\mathfrak{h})$, i.e. that $2\beta, 3\beta-\alpha_1\notin\mathcal{W}$. We argue by contradiction. 

Let us first assume that $2\beta\in\mathcal{W}$. Then there exists $k$, $1\leq k\leq n$, such that $2\beta=\pm\beta\pm\alpha_k$. If $\beta=\pm\alpha_k$ we obtain that $0=\beta\mp\alpha_k$ belongs to $\R(\g)$, which contradicts the property R1 of root systems. If $\beta=\pm\frac{\alpha_k}{3}$, then the roots $\beta+\alpha_k$ and $\beta-\alpha_k$ are proportional, which contradicts R2. 

Now we assume that $3\beta-\alpha_1\in \mathcal{W}$ and conclude similarly. In this case there exists $k$, $1\leq k\leq n$ such that either $2\beta=\alpha_1\pm\alpha_k$ or $4\beta=\alpha_1\pm\alpha_k$. In the first case we obtain $\beta-\alpha_1=-\beta\pm\alpha_k$, which contradicts the fact that roots of $G$ are simple. In the second case \eqref{pscalval} yields $|\beta|^2=\tfrac14 \< \beta,\alpha_1 > \pm \tfrac14 \< \beta,\alpha_k > \le \tfrac18$, which contradicts \eqref{betaval}.

This shows that $2\beta, 3\beta-\alpha_1\in\mathcal{R}(\mathfrak{h})$. Moreover $\<2\beta, 3\beta-\alpha_1>=1$ and thus, by \eqref{newroots}, their difference is a root of $\h$ too: $\beta-\alpha_1=(3\beta-\alpha_1)-(2\beta)\in\mathcal{R}(\mathfrak{h})$, which is in contradiction with $\beta-\alpha_1\in\mathcal{W}$. Consequently, case 1. can not occur.

{\bf Case 2.} From \eqref{normscal} it follows that $|\alpha_i|^2=\frac{5}{4}$, for all $1\leq i\leq 3$. For all $1\leq i,j\leq 3$, $i\neq j$, we then compute: $\<\beta+\alpha_i,\beta+\alpha_j>=\frac{3}{4}+\<\alpha_i,\alpha_j>$, which by \eqref{normscal} must be equal to $0$ or $\pm1$, implying that $\<\alpha_i,\alpha_j>\in \{-\frac{7}{4}, -\frac{3}{4}, \frac{1}{4}\}$. 
On the other hand, we obtain
\[\<\alpha_1,\alpha_2>+\<\alpha_1,\alpha_3>+\<\alpha_2,\alpha_3>=\frac{1}{2}\left(|\beta|^2-\frac{15}{4}\right)=-\frac{7}{4},\]
which is not possible for any of the above values of the scalar products, yielding again a contradiction.

{\bf Case 3.} From \eqref{normscal} it follows that $|\alpha_i|^2=\frac{3}{4}$, for all $1\leq i\leq 3$. Computing the norm of $\beta-\alpha_k=\alpha_i+\alpha_j$, where $\{i,j,k\}$ is any permutation of $\{1,2,3\}$, yields that $\<\alpha_i,\alpha_j>=-\frac{1}{4}$, for all $1\leq i,j\leq 3$, $i\neq j$. We then get
\[\<\beta+\alpha_i,\beta+\alpha_j>=1,\qquad \text{ for all } 1\leq i,j\leq 3, i\neq j,\]
which by the reflexion property \eqref{newroots} implies that 
\beq\label{hh}\{\alpha_i-\alpha_j\}_{1\leq i,j\leq 3}\subseteq\mathcal{R}(\mathfrak{g}).\eeq

We claim that $n=3$ (recall that $n$ denotes the number of vectors $\alpha_i$, or equivalently the quaternionic dimension of $M$). Assume for a contradiction that $n\geq 4$. By \eqref{pscalval}, $\<\beta,\alpha_l>=\frac{1}{4}$ or $\<\beta,\alpha_l>=0$, for any $4\leq l\leq n$. 

If $\<\beta,\alpha_l>=\frac{1}{4}$ for some $l\ge 4$, it follows that $|\alpha_l|^2=\frac{3}{4}$ and $|\beta+\alpha_l|^2=2$, implying by \eqref{normscal} that the scalar product $\<\beta-\alpha_i,\beta+\alpha_l>$   belongs to $\{\pm1,0\}$, for $1\leq i\leq 3$. This further yields that $\<\alpha_i,\alpha_l>\in\{\frac{7}{4}, \frac{3}{4},-\frac{1}{4}\}$. On the other hand, the Cauchy-Schwarz inequality applied to $\alpha_i$  and $\alpha_l$ and the fact that $\mathcal{W}$ has only simple roots (being a root sub-system) imply that the only possible value is $\<\alpha_i,\alpha_l>=-\frac{1}{4}$, for $1\leq i\leq 3$ and $4\leq l\leq n$. Thus, $|\beta+\alpha_l|^2=0$, which contradicts the property R1 of root systems  (\emph{cf.} Definition~\ref{sysroot}). 

We therefore have $\<\beta,\alpha_l>=0$, for all $4\leq l\leq n$. If $|\beta\pm\alpha_l|^2=2$ for some $l\ge 4$ then $|\alpha_l|^2=\tfrac54$,
so $\<\beta-\alpha_l,\beta+\alpha_l>=-\tfrac12$, contradicting
\eqref{normscal}.
Thus $|\beta\pm\alpha_l|^2=1$ for all $4\leq l\leq n$. If $n\ge 5$, \eqref{normscal} implies
$$\<\beta-\alpha_l,\beta+\alpha_s>, \<\beta-\alpha_l,\beta-\alpha_s>\in\left\{0,\pm\frac{1}{2}\right\},\qquad \hbox{for}\ 4\leq l,s\leq n,\ l\neq s.$$ 
This contradicts the equality $\<\beta-\alpha_l,\beta+\alpha_s>+\<\beta-\alpha_l,\beta-\alpha_s>=\frac{3}{2}$, showing that $n\leq 4$.

It remains to show that the existence of $\alpha_4\in\mathcal{A}$, which by the above necessarily satisfies $\<\beta,\alpha_4>=0$ and $|\alpha_4|^2=\frac{1}{4}$, leads to a contradiction. By  \eqref{normscal}, it follows that 
$$1+\<\alpha_i,\alpha_4>=\<\beta+\alpha_i,\beta+\alpha_4>\in\{\pm1,0\},\qquad \forall\ 1\leq i\leq 3.$$ This constraint together with the Cauchy-Schwarz inequality, $|\<\alpha_i,\alpha_4>|\leq \frac{\sqrt{3}}{4}$, implies that $\<\alpha_i,\alpha_4>=0$, for $1\leq i\leq 3$. 

Applying the reflexion property \eqref{newroots} to $\beta+\alpha_4$ and $\beta+\alpha_i$, for $1\leq i\leq 3$, which satisfy $\<\beta+\alpha_i,\beta+\alpha_4>=1$ and $|\beta+\alpha_i|^2=2|\beta+\alpha_4|^2$, it follows that $\alpha_i-\alpha_4, \beta+2\alpha_4-\alpha_i\in \mathcal{R}(\mathfrak{g})$. We now show that all these roots actually belong to $\mathcal{R}(\mathfrak{h})$. Let us assume that $\alpha_i-\alpha_4\in\mathcal{W}$ for some $i\le 3$, i.e. there exists $s$, $1\leq s\leq 4$, such that $\alpha_i-\alpha_4\in\{\pm\beta\pm\alpha_s\}$. Since $\alpha_4$ is orthogonal to $\beta$ and to $\alpha_i$, for $1\leq i\leq 3$, it follows that $\alpha_i-\alpha_4$ must be equal to $\pm\beta-\alpha_4$, leading to the contradiction that $0=\beta\mp\alpha_i\in\mathcal{W}$. Therefore $\alpha_i-\alpha_4\in \mathcal{R}(\mathfrak{h})$. A similar argument shows that $ \beta+2\alpha_4-\alpha_i\in \mathcal{R}(\mathfrak{h})$. 

Now, since the scalar product of these two roots of $\h$ is $\<\alpha_i-\alpha_4, \beta+2\alpha_4-\alpha_i>=-1$, it follows again by \eqref{newroots} that their sum $\beta+\alpha_4$ also belongs to $\mathcal{R}(\mathfrak{h})$, contradicting the fact that $\beta+\alpha_4\in\mathcal{W}(\mathfrak{m})$. This finishes the proof of the claim that $n=3$.

Since the determinant of the Gram matrix $(\<\alpha_i,\alpha_j>)_{1\leq i,j\leq 3}$ is equal to $\frac{5}{16}$, the vectors $\{\alpha_i\}_{1\leq i\leq 3}$ are linearly independent. Thus the roots of $\g$ given by \eqref{hh} can not belong to $\mathcal{W}$, and therefore $\{\alpha_i-\alpha_j\}_{1\leq i,j\leq 3}$ belong to $\mathcal{R}(\mathfrak{h})$. 

Concluding, we have proven that $n=3$ and that the following inclusions hold (after introducing the notation $\gamma_i:=\alpha_j+\alpha_k$ for all permutations $\{i,j,k\}$ of $\{1,2,3\}$):
\begin{equation}\label{rootshg}
\{\gamma_i-\gamma_j\}_{1\leq i\neq j\leq 3}\subseteq\mathcal{R}(\mathfrak{h}),\quad
\{\gamma_i-\gamma_j\}_{1\leq i\neq j\leq 3}\cup\{\pm\gamma_i\}_{1\leq i\leq 3}\subseteq\mathcal{R}(\mathfrak{g}),
\end{equation}
where $\<\gamma_i,\gamma_j>=\delta_{ij}$, for all $1\leq i,j\leq 3$.

Since these sets are closed root systems and we are interested in the representation of $M$ as a homogeneous space $G/H$ with the smallest possible group $G$, we may assume that we have equality in \eqref{rootshg}. Hence $\mathcal{R}(\mathfrak{h})=\{\gamma_i-\gamma_j\}_{1\leq i\neq j\leq 3}$, with $\{\gamma_i\}_{1\leq i\leq 3}$ an orthonormal basis, (which is exactly the root system of the Lie algebra $\su(3)$), and $\mathcal{R}(\mathfrak{g})=\{\gamma_i-\gamma_j\}_{1\leq i\neq j\leq 3}\cup\{\pm\gamma_i\}_{1\leq i\leq 3}$, which is the root system of $\so(7)$. We conclude that the only possible solution is the simply connected homogeneous space $\SO(7)/\U(3)$. 

It remains to check that this space indeed carries a homogeneous almost quaternionic-Hermitian structure. 
Using the sequence of inclusions 
$$\u(3)\subset\so(6)\subset \so(7),$$
we see that the isotropy representation $\m$ of $\SO(7)/\U(3)$ is the direct sum of the restriction to $U(3)$ of the isotropy representation of the sphere $\SO(7)/\SO(6)$, (which is just the standard representation of $\U(3)$ on $\CM^3$), and of the isotropy representation of $\SO(6)/\U(3)$, which is $\Lambda^2(\CM^3)$ (cf. \cite[p. 312]{besse}):
$$\m=\CM^3\oplus\Lambda^2(\CM^3).$$
Let $I$ denote the complex structure of $\m$.
After identifying $\U(1)$ with the center of $\U(3)$ via the diagonal embedding, an element $z\in \U(1)$ acts 
on $\m$ by complex multiplication with $z^3$, i.e. $\iota(z)=z^3$.
Since $\Lambda^2(\CM^3)=(\CM^3)^*$ as complex $\SU(3)$-representations, it follows that the restriction to $\SU(3)$ of the isotropy representation
$\m$ is $\CM^3\oplus(\CM^3)^*$, and thus carries a quaternionic structure, i.e. a complex anti-linear automorphism $J$. We claim that a homogeneous almost quaternionic-Hermitian structure on $\SO(7)/\U(3)$ in the sense of Definition \ref{defi} is given by $\rho:\U(3)\to \SO(3)$
and $\f:\so(3)\simeq{\mathrm{Im}}(\HM)\to\End^-(\m)$ defined by 
$$ \rho(A)=\det(A),   \qquad\f(i)=I,\ \f(j)=J,\ \f(k)=IJ,$$
where $\det(A)\in\U(1)$ is viewed as an element in $\SO(3)$ via the composition $$\U(1)=S(\CM)\to S(\HM)=\Spin(3)\to\SO(3).$$
Indeed, the only thing to check is the equivariance of $\f$, i.e.
\beq\label{u}\f(\rho(A)M\rho(A)^{-1})=\iota(A)\f(M)\iota(A)^{-1},\qquad\forall M\in\so(3),\ \forall A\in\U(3).\eeq
Write $A=zB$ with $B\in\SU(3)$. Then $\rho(A)=z^3$, $\iota(A)=z^3\iota(B)$ and $\iota(B)$ commutes with $I,J,K$, thus with $\f(M)$.
The relation \eqref{u} is trivially satisfied for $M=i$, whereas for $M=j$ or $M=k$ one has $Mz=\bar z M=z^{-1}M$ and similarly $\f(M)\iota(z)\iota(B)=\iota(z^{-1})\iota(B)\f(M)$, so
$$\f(\rho(A)M\rho(A)^{-1})=\f(z^3Mz^{-3})=\f(z^6M)=z^6\f(M)=\iota(z^2)\f(M)=\iota(A)\f(M)\iota(A)^{-1}.$$
This finishes the proof of the theorem.
 \end{proof}
\appendix
\section{Root systems}
For the basic theory of root systems we refer to \cite{a69} and \cite{s}.

\begin{Definition}\label{sysroot}
A set $\mathcal{R}$ of vectors in a Euclidean space $(V,\<\cdot,\cdot>)$ is called a
\emph{root system}  if it satisfies the following
conditions:
\begin{description}
\item[R1] $\mathcal{R}$ is finite, $\mathrm{span}(\mathcal{R})=V$,
  $0\notin \mathcal{R}$.
\item[R2] If $\alpha\in \mathcal{R}$, then the only multiples of
  $\alpha$ in $\mathcal{R}$ are $\pm\alpha$.
\item[R3] $\frac{2\<\alpha,\beta>}{\<\alpha,\alpha>}\in\mathbb{Z}$,
  for all $\alpha, \beta\in \mathcal{R}$.
\item[R4] $s_\alpha:\mathcal{R}\to \mathcal{R}$, for all $\alpha\in
  \mathcal{R}$ ($s_\alpha$ is the reflection $s_\alpha:V\to
  V$, $s_\alpha(v):=v-\frac{2\<\alpha,v>}{\<\alpha,\alpha>}\alpha$).
\end{description}
\end{Definition}

Let $G$ be a compact semi-simple Lie group with Lie algebra $\mathfrak{g}$ 
endowed with an $\ad_\gg$-invariant scalar product. 
Fix a Cartan sub-algebra $\tt\subset\gg$ and let $\mathcal{R}(\gg)\subset \mathfrak{t}^*$ denote
its root system. It is well-known that $\mathcal{R}(\gg)$ satisfies the conditions 
in Definition \ref{sysroot}. Conversely,
every set of vectors satisfying the conditions in Definition \ref{sysroot} is the root system of 
a unique semi-simple Lie algebra of compact type.

\begin{Remark}[Properties of root systems]
Let $\mathcal{R}$ be a root system. If
$\alpha,\beta\in\mathcal{R}$ such that $\beta\neq\pm\alpha$ and
$\norm(\beta)\geq\norm(\alpha)$, then either $\<\alpha,\beta>=0$
or 
\begin{equation}\label{normscal}
\left(\frac{\norm(\beta)}{\norm(\alpha)},
  \frac{2\<\alpha,\beta>}{\<\alpha,\alpha>}\right)\in\{(1,\pm
1),(2,\pm 2),(3,\pm 3)\}.
\end{equation}
In other words, either the scalar product of two roots vanishes, or its absolute value equals half the square length of the longest root.
Moreover,
\begin{equation}\label{newroots}
\beta-\mathrm{sgn}\left(\frac{2\<\alpha,\beta>}
  {\<\alpha,\alpha>}\right)k\alpha\in\mathcal{R},  \quad \text{ for all }
k\in\mathbb{Z}, 1\leq k\leq \biggl|\frac{2\<\alpha,\beta>}{\<\alpha,\alpha>}\biggr|.
\end{equation}
\end{Remark}

\begin{Definition}[\cite{mp}]\label{asubsys}
A set $\mathcal{P}$ of vectors in a Euclidean space $(V,\<\cdot,\cdot>)$
is called a \emph{root sub-system} if it satisfies the conditions R1 - R3 from
Definition \ref{sysroot} and if the set $\overline{\mathcal{P}}$ obtained from $\mathcal{P}$ by taking all possible reflections
is a root system. 
\end{Definition}


\begin{thebibliography}{5}

\bibitem{a69}
J. Adams, {\it Lectures on Lie groups}, The University of Chicago Press, Chicago (1969).

\bibitem{besse} A. Besse, {\it Einstein manifolds}, Ergebnisse der Mathematik und ihrer Grenzgebiete (3)  {\bf 10}
 Springer-Verlag, Berlin, 1987.

\bibitem{bs49} A. Borel, J. de Siebenthal, {\sl  Les sous-groupes ferm\'es de rang maximum des groupes de
Lie clos}, Comment.\ Math.\ Helv.\ {\bf 23}  (1949), 200--221.

\bibitem{j1} D. Joyce, {\sl The hypercomplex quotient and the quaternionic quotient}, Math.\ Ann.\ {\bf 290} (1991), 323--340.

\bibitem{j2} D. Joyce, {\sl Compact hypercomplex and quaternionic manifolds}, J.\ Differ.\ Geom.\ {\bf 35} (1992), 743--761. 

\bibitem{ma10} O. Maci\'a, {\sl A nearly quaternionic structure on $\SU(3)$}, J.\ Geom.\ Phys.\
{\bf 60} (2010), no.\ 5, 791--798.

\bibitem{ca04} F. Mart\'\i n Cabrera, {\sl Almost Quaternion-Hermitian Manifolds},
Ann.\ Global Anal.\ Geom.\ {\bf 25} (2004), 277--301. 

\bibitem{cs04} F. Mart\'\i n Cabrera, A. F. Swann, {\sl Almost Hermitian structures and quaternionic geometries},
Differential Geom.\ Appl. {\bf 21} (2004), no.\ 2, 199--214.

\bibitem{cs08} F. Mart\'\i n Cabrera, A. F. Swann, {\sl The intrinsic torsion of almost
quaternion-Hermitian manifolds}, Ann.\ Inst.\ Fourier {\bf 58} (2008), no.\ 5, 1455--1497. 

\bibitem{cs12} F. Mart\'\i n Cabrera, A. F. Swann, {\sl Quaternion Geometries on the Twistor Space of
the Six-Sphere}, arXiv:1302.6397.

\bibitem{mp}  A. Moroianu, M. Pilca, {\sl Higher Rank Homogeneous Clifford structures}, 
J.\ London Math.\ Soc.\ {\bf 87} (2013), 384--400.
    
\bibitem{ms}  A. Moroianu, U. Semmelmann, {\sl Clifford structures on
    Riemannian manifolds}, Adv.\ Math.\ {\bf 228} (2011), 940--967.

\bibitem{au} A. Moroianu, U. Semmelmann, {\sl Weakly complex homogeneous spaces},
J.\ reine angew.\ Math.\ (2014), doi:10.1515/crelle-2012-0077.

\bibitem{ff} A. Moroianu, U. Semmelmann, {\sl Invariant four-forms and symmetric pairs},
Ann.\ Global Anal.\ Geom. {\bf 43}  (2013), 107--121.

\bibitem{sa82} S. M. Salamon, {\sl Quaternionic K\"ahler manifolds}, Invent.\ Math. {\bf 67} (1982), 143--171.

\bibitem{s} H. Samelson, {\it Notes on Lie Algebras}, Springer-Verlag (1990).

\bibitem{sw97} A. F. Swann, {\sl Some Remarks on Quaternion-Hermitian Manifolds}, Arch.\ Math.\ {\bf 33} (1997), 349--354.

 \bibitem{w} J. A. Wolf, {\sl Complex homogeneous contact manifolds and quaternionic symmetric spaces}, J.\ Math.\ Mech.\ {\bf 14} (1965), 1033--1047.

\end{thebibliography}
\end{document}